\documentclass[10pt]{amsart}
\usepackage{mathrsfs}
\usepackage{amsfonts}
\usepackage{color}
\usepackage{latexsym,amsmath,amssymb}
\usepackage{hyperref}
\newtheorem{theorem}{Theorem}[section]
\newtheorem{lemma}[theorem]{Lemma}

\renewcommand{\thetheorem}{\arabic{section}.\arabic{theorem}}
\theoremstyle{definition}

\theoremstyle{remark}

\numberwithin{equation}{section}
\begin{document}
\title[Classification results of Liouville equations]{ Classification results of Liouville equations and rigidity of Riemannian surfaces}

\author{Qianzhong Ou}
\address{ Department of Mathematics\\
          Yunnan Normal University\\
          Kunming, Yunna, 650500, P.R. China\\
          Yunnan Key Laboratory of Modern Analytical Mathematics and Applications, Kunming, China.}
\email{ouqzh@163.com}


\maketitle

\begin{abstract}
 We study  the Liouville equation $\triangle u+e^{2u} =0$ in a Riemannian surface $(M, g)$ with nonnegative $Ricci$ curvature.
 Under some asymptotic lower bound assumptions, we classify all the solutions to this equation, meanwhile we obtain the rigidity results for the ambient manifold.
 Note that our assumptions are optimal in some sense and differ from the classical assumption of finite total curvature.

\end{abstract}




\section{Introduction}

\setcounter{equation}{0}

 Consider the following Liouville equation

\begin{equation}\label{1.1}
  \triangle u +e^{2u}=0         \qquad \text{in} \quad (\mathbb{R}^2,g_0),
\end{equation}
where ($\mathbb{R}^2,g_0$) denotes the Euclidean plane with the standard flat Euclidean metric $g_0$ which may be omitted without confusion.
  This equation arises from the study of conformal metrics $e^{2u}g_0$ which has constant Gauss curvature 1.
This problem goes back to the work of J. Liouville \cite{Lio1853} who found that solutions to (\ref{1.1}) can be represented
by some meromorphic functions on the complex plane $\mathbb{C}$.

It is well known that if there is no any further assumption on the solutions, there are infinitely many solutions of (\ref{1.1}).
 Then, a classical result was given by the seminal work \cite{CL1991} of Chen and Li, where they classified all solutions of (\ref{1.1})
 under the assumption of finite total curvature, i.e., $\int_{\mathbb{R}^2}e^{2u}< \infty$. Precisely, they showed in \cite{CL1991} that any solution of (\ref{1.1})
with finite total curvature must be given by

\begin{equation}\label{1.2}
  u(x)=-\ln(a+b|x-x_0|^2)
\end{equation}
 for some $a,b>0$ with $4ab=1$ and $x_0\in \mathbb{R}^2$.

 In this paper, we are concerned about classification results to this Lioville equation without the assumption of finite total curvature.
 Moreover, we shall extend the classification results to the case of Riemannian surface and at the same time, obtain the rigidity of the underlying manifold.
 More precisely, we have the following theorem.

\begin{theorem}\label{Thm1}
Let $(M,g)$ ) be a complete, connected, non-compact, boundaryless Riemannian manifold of dimension 2 and
with nonnegative $Ricci$ curvature. Let $F(t)$ be any positive nondecreasing function satisfying
            $\int_c^{+\infty}\frac{dt}{tF(t)}=+\infty$
for some $c>1$. Denotes by $r=r(x)$ the geodesic distance from a fixed reference point of $M$. If $u\in C^3(M)$ is a solution to

\begin{equation}\label{1.3}
  \triangle u +e^{2u}=0  \qquad \text{in} \quad M,
\end{equation}
such that

 \begin{equation}\label{1.4}
  u(x)\geqslant -2\ln \big[ r(x)F^{\alpha}(r(x))\big]
\end{equation}
for $r(x)>c$ and  some arbitrary $\alpha\geqslant 0$, then $(M,g)$ must be isometric to $(\mathbb{R}^2, g_0)$ and $u$ is given by (\ref{1.2}).

\end{theorem}

 Note that the above result is optimal in the sense that $\alpha\geqslant 0$ is arbitrarily while for any $\beta>2$,
 there is a non flat Riemannian manifold ($M,g$) with nonnegative $Ricci$ curvature which admits solutions to the
 Liouville equation  (\ref{1.3}) satisfying $u(x)\geqslant -\beta\ln \big[ r(x)\big]$ for $r(x)$ large (the explicit examples are provided
on pp.34 of \cite{CL2024} ). Note also that the classification result is new even in the case  ($M,g$)=($\mathbb{R}^2,g_0$).

Recently, there appeared some interesting results  on the classifications of solutions to
 (\ref{1.1}) and (\ref{1.3}) without the assumption of finite total curvature.
 Ciraolo-Farina-Polvara \cite{CFP2024} give the same results of theorem \ref{Thm1} under the assumption (\ref{1.4}) with $\alpha=\frac{1}{2}$.
Eremenko-Gui-Li-Xu \cite{EGLX2022} give a complete classification of solutions to (\ref{1.1}) which are bounded from above.
Other classification results are also given in \cite{EGLX2022} without the assumption $\int_{M}e^{2u}dg< \infty$,
including results under the asymptotic upper bound assumption or the concavity assumptiom.
For a good survey of classification results for solutions to the Liouville equation, one can see the reference \cite{CaiL2024} by Cai and Lai.

We would like to point out that the asymptotic lower bound assumptions analogous to (\ref{1.4}) have appeared in \cite{CM2022} and \cite{CL2024},
although they both preserved the finite total curvature assumption $\int_{M}e^{2u}dg< \infty$.
The asymptotic lower bound assumption  seams interesting and reasonable.
As already noted, in (\ref{1.4}), the coefficient $-2$ is optimal.
Recall that, the analogous asymptotic assumption $u(x)\sim O(|x|^{2-n})$ ($|x|\rightarrow +\infty$) was posed in the pioneer work \cite{GNN1979}
by Gidas-Ni-Nirenberg, in studying the higher dimensional Yamabe equation in $\mathbb{R}^n$ ($n\geqslant 3$).
One may also notice that  $u(x)=-2\ln r(x) $ and $u(x)=|x|^{2-n} $
are the fundamental solutions of $\triangle u=0$ in $\mathbb{R}^2$ and in $\mathbb{R}^n$ ($n\geqslant 3$), respectively. On the other hand,
Theorem \ref{Thm1} also implies a necessarily condition for the solvability of the equation (\ref{1.3}), that is,
any solution $u\in C^3(M)$ of (\ref{1.3}) must satisfies $\liminf_{r(x)\rightarrow +\infty}\frac{u(x)}{\ln r(x)}\leqslant -2$.
Since all solutions of (\ref{1.1}) bounded from above are completely  understood by Eremenko-Gui-Li-Xu \cite{EGLX2022},
this necessarily condition also implies that any unknown solution of (\ref{1.1}) must be neither bounded form above nor bounded from below.

Since there are infinitely many solutions of (\ref{1.1}) and the classification problem is very complicate and challenging and
 is far from completely understood, various assumptions on the solutions have been proposed.  Although the assumption  (\ref{1.4})
 is optimal as mentioned before, with a weaken version of it, we can still obtain a new classification result as follows.

\begin{theorem}\label{Thm2}
Let $(M,g)$ be as in Theorem \ref{Thm1}. If $u\in C^3(M)$ is a solution to (\ref{1.3}) such that

 \begin{equation}\label{1.5}
  u(x)\geqslant -\beta\ln \big[ r(x)\big]
\end{equation}
for $r(x)>c$ and some arbitrary $\beta> 2$, then $(M,g)$ must be conformal to $(\mathbb{R}^2, g_0)$.

\end{theorem}

Under the assumption (\ref{1.5}) with $\beta=3$, the result of Theorem \ref{Thm2} was also obtained by \cite{CFP2024}.
Once again, the result here is optimal in the sense that $\beta> 2$ is arbitrarily while there are actually
complete manifolds $(M, g)$ of dimension $2$ with nonnegative $Ricci$ curvature, conformal to $\mathbb{R}^2$ with
the Euclidean metric which admits solutions $u(x)$ to the Liouville equation (\ref{1.3}) verifying
$u(x)\sim -\beta\ln \big[ r(x)\big]$ as $r(x)\rightarrow \infty$ for any  $\beta>2$(see also pp.34 of \cite{CL2024} for the explicit examples ).

We will prove Theorem \ref{Thm1} and Theorem \ref{Thm2} through a strategy adopted by \cite{CFP2024}, except with an alternative
theorem provided in \cite{K1982} as follows.

\setcounter{theorem}{0}
\renewcommand{\thetheorem}{\Alph{theorem}}

\begin{theorem}\label{Thm3}{\bf (Theorem 2.4 in \cite{K1982})}
If $P$ is a $C^2$ nonconstant positive solution of $\triangle \ln P \geqslant 0$ on a complete noncompact Riemannian manifold $(M^n,g)$
then for any $q>0$, $x_0\in M^n$, and function $F(t)$ as in theorem \ref{Thm1} we have

\begin{equation}\label{1.6}
  \limsup_{r\rightarrow \infty} \frac{1}{r^2F(r)}\int_{B_r(x_0)} P^q dg=+\infty,
\end{equation}
where $B_r(x_0)$ denote the geodesic balls in $(M^n,g)$.

\end{theorem}
We would like to point out that the key idea of this strategy is the so called $P$-function method, which may first due to Weinberger \cite{We1971}
and later applied by Wang \cite{W2022}. This approach  was also adopted very recently by Sun-Wang \cite{SW2025}
to deal with a higher dimensional quasilinear Liouville equation. One of the key steps of our proofs is  verifying
the condition $\triangle \ln P \geqslant 0$ in Theorem \ref{Thm3}, which will be given in section 2 (see Lemma \ref{lem-1}).
Then the proofs of Theorem \ref{Thm1} and Theorem \ref{Thm2} shall be presented following Lemma \ref{lem-1} in the same section.

\section{Proofs of Theorem \ref{Thm1} and Theorem \ref{Thm2}}

\setcounter{equation}{0}
\setcounter{theorem}{0}
\renewcommand{\thetheorem}{\arabic{section}.\arabic{theorem}}

In this section, we will prove Theorem \ref{Thm1} and Theorem \ref{Thm2}.
To do this, we first present a preparation lemma which is a key step to our proofs.

Taking $v=e^{-u}$ for $u$ being any solution of the Liouville equation (\ref{1.3}), then

\begin{equation}\label{2.1}
  \triangle v= |\nabla v|^2v^{-1}+v^{-1} \qquad \text{in} \quad M.
\end{equation}
Denote $P:= |\nabla v|^2v^{-1}+v^{-1}$, then $\triangle v=P$. Denote the vector fields
$E_{ij}= v_{,ij}-\frac{\triangle v}{2}g_{ij}$. Then clearly the  matrix $E=\{E_{ij}\}$ is symmetric and trace free, i.e., $\text{\bf Tr}_g E\equiv 0$.
Here and in the sequel, we would indicate the covariant derivatives of functions or vector fields with indices preceded by a
comma, to avoid confusion, and, the summation convention for repeated indices are used.
For the matrix $g=\{g_{ij}\}$, we denote $g^{-1}=\{g^{ij}\}$, the inver of $g$, and we also use them to raise or lower the indices.
 For examples, $|\nabla v|^2=v^{,i}v_{,i}=g^{ij}v_{,i}v_{,j}$ and $|E|^2=E_{ij}E^{ij}=g^{ik}E_{ij}E_{kl}g^{lj}$.

 Now we have the following key lemma for the proofs of the main theorems.

 \begin{lemma}\label{lem-1}
With the notations as in above, then  we have
\begin{equation}\label{2.5}
 \bigtriangleup(\ln P)=2v^{-2}P^{-2}|E|^2+2v^{-1}P^{-1}\text{Ric}(\nabla v,\nabla v).
\end{equation}
\end{lemma}

{\bf Proof}\quad  First we have

\begin{equation}\label{2.6}
  \triangle (\ln P)=P^{-1}\triangle P -P^{-2}|\nabla P|^2.
\end{equation}
Then we compute

\begin{equation}\label{2.7}
\begin{split}
\triangle P=& \triangle\big(|\nabla v|^2v^{-1}+v^{-1}\big)\\
           =& (|\nabla v|^2+1)\triangle\big(v^{-1}\big)+2\big< \nabla(v^{-1}),\nabla(|\nabla v|^2+1)\big>+v^{-1}\triangle \big(|\nabla v|^2+1\big)\\
           =& (|\nabla v|^2+1)\big(2v^{-3}|\nabla v|^2-v^{-2}\triangle v\big)+2\big< \nabla(v^{-1}),\nabla(vP)\big>+v^{-1}\triangle \big(|\nabla v|^2\big)\\
           =& P\big(2v^{-2}|\nabla v|^2-v^{-1}\triangle v\big)+2\big< -v^{-2}\nabla v,P\nabla v+v\nabla P\big>+v^{-1}\triangle \big(|\nabla v|^2\big)\\
           =& -v^{-1}P\triangle v-2v^{-1}\big< \nabla v,\nabla P\big>+v^{-1}\triangle \big(|\nabla v|^2\big),
\end{split}
\end{equation}
where $\langle\cdot, \cdot\rangle$ denotes the inner product with respect to $g$.
By the Bochner formula we have

\begin{equation}\label{2.8}
  \triangle \big(|\nabla v|^2\big)=2|\nabla^2 v|^2 +2\big< \nabla v,\nabla\triangle v\big>+ 2\text{Ric}(\nabla v,\nabla v).
\end{equation}
Plugging this into (\ref{2.7}) and using the equation $\triangle v=P$ we get

\begin{equation}\label{2.9}
\begin{split}
\triangle P=& 2v^{-1}|\nabla^2 v|^2 -v^{-1}P\triangle v+ 2v^{-1}\text{Ric}(\nabla v,\nabla v)\\
           =& 2v^{-1}|\nabla^2 v- \frac{\triangle v}{2}g|^2 + 2\text{Ric}(\nabla v,\nabla v).
\end{split}
\end{equation}
Inserting this into (\ref{2.6}) yields

\begin{equation}\label{2.10}
\begin{split}
\triangle (\ln P)=& 2v^{-1}P^{-1}|\nabla^2 v- \frac{\triangle v}{2}g|^2 -P^{-2}|\nabla P|^2+ 2v^{-1}P^{-1}\text{Ric}(\nabla v,\nabla v)\\
                 =& 2v^{-2}P^{-2}\big(|E|^2 + |\nabla v|^2|E|^2-\frac{1}{2}v^2|\nabla P|^2\big)+ 2v^{-1}P^{-1}\text{Ric}(\nabla v,\nabla v).
\end{split}
\end{equation}
A direct calculation will show that $|\nabla v|^2|E|^2-\frac{1}{2}v^2|\nabla P|^2\equiv 0$.
In fact, we may assume $g_{ij}=\delta_{ij}$ at any fixed point, then

\begin{equation}\label{2.11}
\begin{split}
\nabla  P=& \nabla\big(|\nabla v|^2v^{-1}+ v^{-1}\big)\\
         =& v^{-1}\nabla\big(|\nabla v|^2+ 1\big) +(|\nabla v|^2+ 1)\nabla\big(v^{-1}\big)\\
         =& 2v^{-1}\nabla v\otimes \nabla^2v -v^{-1}P\nabla v \\
         =& 2v^{-1}\nabla v\otimes (\nabla^2v -\frac{P}{2}g)\\
         =& 2v^{-1}\nabla v\otimes E,
\end{split}
\end{equation}
and hence

\begin{equation}\label{2.12}
\begin{split}
|\nabla v|^2|E|^2-\frac{1}{2}v^2|\nabla P|^2
=& |\nabla v|^2|E|^2 -2|\nabla v\otimes E|^2\\
=& \big((v_{,2})^2-(v_{,1})^2\big)(E_{11}-E_{22})(E_{11}+E_{22})\\
\,&-2v_{,1}v_{,2}E_{12}(E_{11}+E_{22})\\
=& 0.
\end{split}
\end{equation}
where in the last step we have used $\text{\bf Tr}_g E=E_{11}+E_{22} \equiv 0$. Now we have verified (\ref{2.5}) by (\ref{2.10}).\qed

Now we are at the point to prove Theorem \ref{Thm1} and Theorem \ref{Thm2}.

\vspace{0.6cm}
{\bf Proof of Theorem \ref{Thm1} }

Let $B_r(x_0)$ be geodesic balls in $(M^n,g)$ and $\eta\in C^2(M)$ be the standard cutoff functions such that

\begin{equation}\label{2.13}
\begin{cases}
                          \eta\equiv 1  &\verb"in" \,\,B_r(x_0),\\
                        0\leq\eta\leq1  &\verb"in" \,\,B_{2r}(x_0),\\
                          \eta\equiv 0  &\verb"in" \,\,M\backslash B_{2r}(x_0),\\
     |\nabla \eta|\leqslant \frac{1}{r}  &\verb"in" \,\,M.
\end{cases}
\end{equation}
Multiplying both sides of the equation (\ref{2.1}) by $\eta^2$  and integrating (by part) over $M$ we have

\begin{equation}\label{2.14}
\begin{split}
\int_M \eta^2 |\nabla v|^2v^{-1} + \int_M \eta^2 v^{-1 }
       = & \int_M \eta^2 \triangle v\\
       = & -2\int_M \eta\nabla\eta\cdot\nabla v \\
\leqslant & \frac{1}{2}\int_M \eta^2 |\nabla v|^2v^{-1} +2\int_M |\nabla\eta|^2 v,
\end{split}
\end{equation}
where the Young's inequality being used in the last step. Then we get

\begin{equation}\label{2.15}
\begin{split}
\int_{B_r(x_0)} P \leqslant & \int_M \eta^2\big(|\nabla v|^2v^{-1} + v^{-1}\big)\\
\leqslant & 4\int_M |\nabla\eta|^2 v\\
\leqslant &  \frac{4}{r^2}\int_{B_{2r}(x_0) \backslash B_r(x_0)}  v.
\end{split}
\end{equation}
From the assumption (\ref{1.4}) we have

\begin{equation}\label{2.16}
v(x)=e^{-u(x)}\leqslant r^2(x)F^{2\alpha}(r(x)) \quad \text{for}\quad r(x)>c.
\end{equation}
Then by (\ref{2.15})

\begin{equation}\label{2.17}
\int_{B_r(x_0)} P
\leqslant  16|B_{2r}(x_0)| F^{2\alpha}(2r)\quad \text{for}\quad r>c,
\end{equation}
where $|B_{2r}(x_0)|$ denotes the volume of $B_{2r}(x_0)$ with respect to $g$. Therefore, by H\"{o}lder inequality

\begin{equation}\label{2.18}
\begin{split}
\int_{B_r(x_0)} P^{\frac{1}{2\alpha}}
\leqslant & \big(\int_{B_r(x_0)} P\big)^{\frac{1}{2\alpha}} \big(\int_{B_r(x_0)} \big)^{\frac{2\alpha-1}{2\alpha}}\\
\leqslant & \big(16 |B_{2r}(x_0)| F^{2\alpha}(r(x))\big)^{\frac{1}{2\alpha}} |B_r(x_0)|^{\frac{2\alpha-1}{2\alpha}}\\
        = & 4^{\frac{1}{\alpha}} |B_{2r}(x_0)| F(2r),
\end{split}
\end{equation}
where we have assumed $2\alpha>1$ without loss generality. Since Ric$\geqslant 0$, the Bishop-Gromov relative volume comparison theorem implies
$|B_r(x_0)|\leqslant r^2D$, $\forall r>0$, where $D$ is the Euclidean area of the unit disk in $\mathbb{R}^2$. Then

\begin{equation}\label{2.19}
\frac{1}{r^2\widetilde{F}(r)}\int_{B_r(x_0)} P^{\frac{1}{2\alpha}}\leqslant 4^{\frac{1}{\alpha}+1} D,
\end{equation}
where $\widetilde{F}(r)=F(2r)$, which satisfies
$\int_c^{+\infty}\frac{dt}{t\widetilde{F}(t)}=\int_c^{+\infty}\frac{2dt}{2tF(2t)}=\int_{2c}^{+\infty}\frac{dt}{tF(t)}=+\infty$
by the assumption of $F$. Now by (\ref{2.19}) we have

\begin{equation}\label{2.20}
\limsup_{r\rightarrow \infty} \frac{1}{r^2\widetilde{F}(r)}\int_{B_r(x_0)} P^{\frac{1}{2\alpha}}\leqslant 4^{\frac{1}{\alpha}+1} D.
\end{equation}

On the other hand, clearly $P>0$ and if Ric$\geqslant 0$, by Lemma \ref{lem-1} we have $\triangle \ln P \geqslant 0$. We can thus apply Theorem \ref{Thm3}
to obtain $P=constant$. Then again, using Lemma \ref{lem-1} yields

 $$E\equiv 0 \,\, \text{i.e.}\,\,  \nabla^2 v\equiv \frac{P}{2}g \quad \text{and} \quad\text{Ric}(\nabla v,\nabla v)\equiv 0.$$

\noindent Then  $(M,g)$ must be the Euclidean plane $\mathbb{R}^2$ by a classical splitting theorem (see e.g. Theorem 5.7.4 in \cite{P2006} ),
and further more, $v$ is quadratic, that is, $u$ must be given by (\ref{1.2}). \qed

\vspace{0.6cm}

{\bf Proof of Theorem \ref{Thm2} }

Under the assumption of ($M,g$), the  classical results of Cohn-Vossen \cite{CV1935} and
Huber \cite{H1957}  imply  that either ($M,g$) is isometric to the flat cylinder $S^1\times \mathbb{R}$
or ($M,g$) is conformal to ($\mathbb{R}^2, g_0$). Next, we need only to  rule out the flat cylinder under the assumption (\ref{1.5}).
For contradiction, suppose that ($M,g$) is the flat cylinder, then ($M,g$) has linear volume growth, i.e., there exists some constant
$C>0$ such that $|B_r|\leqslant Cr$ , for large radii $r$. Then under the assumption (\ref{1.5}) we have

\begin{equation}\label{2.21}
v(x)=e^{-u(x)}\leqslant r^{\beta}(x) \quad \text{for}\quad r(x)>c,
\end{equation}
and by (\ref{2.15}),

\begin{equation}\label{2.22}
\int_{B_r(x_0)} P\leqslant \frac{4}{r^2}\int_{B_{2r}(x_0) \backslash B_r(x_0)}  v
\leqslant   \frac{4}{r^2}|B_{2r}(x_0)| (2r)^{\beta}\quad \text{for}\quad r>c.
\end{equation}
Now similar to (\ref{2.18}) we get

\begin{equation}\label{2.23}
\begin{split}
\int_{B_r(x_0)} P^{\frac{1}{\beta}}
\leqslant & \big(\int_{B_r(x_0)} P\big)^{\frac{1}{\beta}} \big(\int_{B_r(x_0)} \big)^{\frac{\beta-1}{\beta}}\\
\leqslant & \big(\frac{4}{r^2}|B_{2r}(x_0)| (2r)^{\beta}\big)^{\frac{1}{\beta}} |B_r(x_0)|^{\frac{\beta-1}{\beta}}\\
        = & 4^{\frac{1}{\beta}}r^{-\frac{2}{\beta}} (2r) |B_r(x_0)|\Big( \frac{|B_{2r}(x_0)|}{|B_r (x_0)|}\Big)^{\frac{1}{\beta}}\\
        = & C2^{1+\frac{3}{\beta}} r^{2-\frac{2}{\beta}}.
\end{split}
\end{equation}
Therefore

\begin{equation}\label{2.24}
\limsup_{r\rightarrow \infty} \frac{1}{r^2}\int_{B_r(x_0)} P^{\frac{1}{\beta}}\leqslant 0.
\end{equation}
As in the previous proof of Theorem \ref{Thm1}, we can conclude that $P$ must be constant
and hence $(M,g)$ must be isometric to ($\mathbb{R}^2, g_0$). A contradiction.\qed

\vspace{0.6cm}
{\it Acknowledgement.}
 Research of the author  was supported by National Natural Science Foundation of China (grants 12471194 and 12141105).

\end{document}